\documentclass[11pt]{article}
\usepackage{amsmath, amsthm}
\usepackage{amsmath}
\usepackage{amsfonts}
\usepackage{amssymb}
\usepackage{mathrsfs} 
\usepackage{verbatim}
\usepackage[pdftex]{graphicx}
\usepackage{epsfig}
\usepackage{amscd}
\usepackage{amsthm}
\usepackage{color}

\numberwithin{equation}{section}
\newtheorem*{mainthm}{Main Theorem}

\newtheorem{thm}{Theorem}[section]
\newtheorem{cor}[thm]{Corollary}
\newtheorem{lemma}[thm]{Lemma}
\newtheorem{prop}[thm]{Proposition}

\newtheorem{remark}[thm]{Remark}

\theoremstyle{definition}

\newtheorem{defn}[thm]{Definition}
\newtheorem*{remmark}{Remark}

\newcommand{\R}{\mathbb{R}}
\newcommand{\e}{{\varepsilon}}

\begin{document}
\title{\Large{Degenerate diffusion with a drift potential: a viscosity solutions approach, joint work with I.~C.~Kim, truncated version}}
\author{H.~K.~Lei$^1$}
\vspace{-2cm}
\date{}
\maketitle
\vspace{-4mm}
\centerline{${}^1$\textit{Department of Mathematics, California Institute of Technology}}
 
\begin{quote}
{\footnotesize {\bf Abstract: } This is a truncated version of the paper \emph{Degenerate diffusion with a drift potential: a viscosity solutions approach}, co--authored with I.~C.~Kim.  The purpose of this version is to withdraw the claim of quantitative rate of convergence of the free boundary on the part of H.~K.~Lei.  The difference from the previous version lies in Section 3 where 1) the quantitative version of the convergence of the free boundary statement has been removed and 2) the more basic version of some convergence of the free boundary given uniform convergence of the function has been rewritten.}
\end{quote}

\begin{quote}
{\footnotesize We introduce a notion of viscosity solution for a  nonlinear degenerate diffusion equation with a drift potential. We show that our notion of solution coincides with that of the weak solution defined via integration by parts. As an application of the viscosity solutions theory, we show that in the case of a strictly convex potential, the free boundary uniformly converges to equilibrium as $t$ grows.}
\end{quote}

\section{\large{Introduction}}

Consider a $C^2$--potential $\Phi(x):\R^n\to \R$, and consider a nonnegative, continuous function $\rho_0(x):\R^n\to \R$ which has compact support $\Omega_0$. In this paper  we study the porous medium equation with a drift
\begin{equation} \label{main_eq_density}\rho_t = \Delta(\rho^m) + \nabla \cdot (\rho \nabla \Phi),\end{equation}
for $m > 1$, with initial data $\rho_0(x)$.  It will be convenient to change from the density variable to the pressure variable 
\begin{equation}\label{density}
u = \frac{m}{m-1}\rho^{m-1}, \quad u_0 =\frac{m}{m-1}\rho_0^{m-1},
\end{equation}
so that the equation becomes 
$$
u_t = (m-1) u \Delta u + |\nabla u|^2 + \nabla u \cdot \nabla \Phi + (m-1) u \Delta \Phi\leqno\mbox{(PME--D)}
$$
(for more on the density to pressure transform see e.g., the discussions in \cite{BV}).  We consider continuous and nonnegative solutions in the space--time domain $Q = \mathbb R^n \times (0, T)$ for some $T > 0$, with prescribed initial conditions $u(x, 0) = u_0(x) \in C(\mathbf R^n)$.  

When $\Phi\equiv  0$, (PME--D) is the widely-studied Porous Medium Equation (PME): We refer to the book \cite{V} for references. 
Moreover, when $V=|x|^2$, (PME--D) is obtained as a re-scaled version of the (PME) via the transform
$$
\theta(\eta,\tau):=t^{-\alpha}u(x,t), \eta = xt^{-\beta} , \tau = \ln t
$$
(here $u$ solves (PME)). This suggests that the local behavior of (PME--D) is similar to that of (PME), with perturbations due to the inhomogeneity of $\Phi$. We will illustrate this fact in the construction of various barriers in Section 2.

The weak solution theory for (PME--D) in the case of bounded domains has been developed in \cite{BH} and \cite{DiB}.  Also, in \cite{JGJ}, existence and uniqueness of solutions are established for the full space case under reasonable assumptions (either the initial data is compactly supported or the potential has less than quadratic growth at infinity). 

Further, uniform convergence to equilibrium for (PME--D) has also been shown in \cite{BH} (see Theorem 3.1).  In \cite{JT}, the connection between the (PME) and the nonlinear Fokker--Planck equation is established, which facilitates the use of the entropy method to derive an explicit $L^1$ rate of convergence.  In \cite{jose_et_al}, an extensive study is made of a general form of the nonlinear Fokker--Planck equation, i.e., $\rho_t = \nabla \cdot (\nabla\varphi(\rho) + \rho \nabla V)$ with suitable assumptions on $\varphi$ and exponential $L^1$ rate of convergence is obtained.  (PME--D) falls under the framework of \cite{jose_et_al}, and in fact it is the case that almost all of our results would also go through for a general equation of this form, but for ease of exposition we will restrict attention to (PME--D).

We introduce a notion of viscosity solution  for the free boundary problem associated with this equation, which we will show to be equivalent to the usual notion of weak solutions -- see \cite{CIL} for the general theory of viscosity solutions. Note that, formally, the  {\it free boundary}  
\[\Gamma(u):= \partial\{u>0\}\]
 moves with the outward normal velocity
$$
V= \frac{u_t}{|\nabla u|} = (\nabla u+\nabla\Phi)\cdot \frac{\nabla u}{|\nabla u|} = |\nabla u| +\nabla \Phi\cdot\frac{\nabla u}{|\nabla u|},
$$
where the first equality is due to the fact that $u=0$ on $\Gamma(u)$.  In this regard we closely follow the framework and arguments set out in \cite{CV} (see also \cite{Kim} and \cite{BV}), where the viscosity concept is introduced and studied for the Porous Medium Equation.  We point out especially that \cite{BV} extends the result of \cite{CV} to the case where the diffusion term is multiplied by more general nonlinearities; our focus, however, is on the added drift term, which introduces spatial inhomogeneities.  The key utility of the viscosity concept here is that we will be able to describe the pointwise behavior of the free boundary evolution by maximum principle arguments with local barriers. As an application, we are able to extend the results of \cite{BH} and \cite{jose_et_al} to a stronger notion of free boundary convergence.  We summarize our main result in the following:

\begin{mainthm}\label{main_thm}
There exists a viscosity solution $u$ of (PME--D) in the sense of Definition \ref{viscosity} and 
\begin{itemize}
\item[(a)] $u$ is unique and coincides with the weak solution studied in \cite{BH} and \cite{jose_et_al}.\\ 
\item[(b)] Suppose $\Phi$ is strictly convex, then we have convergence of the free boundary:
$\Gamma(u(\cdot, t))$ uniformly converges to $\Gamma(u_\infty)$, where $u_\infty$ is the unique equilibrium solution to which $u$ tends as $t \rightarrow \infty$ (see Theorem \ref{equilibrium_solution}).
\end{itemize}
\end{mainthm}

We will separately (re)state and prove items (a) and (b) in the above as Theorem \ref{exists_unique} and Theorem \ref{conv_fb}.

\vspace{10pt}

\begin{remmark}
~
\begin{itemize}
\item[$\circ$]The free boundary convergence may not hold if $|\nabla \Phi|$ vanishes at some points, even though the uniform convergence of the solution still holds.  This is the reason for assuming $\Phi$ to be strictly convex.

\item[$\circ$]In the case of $\Phi(x)=|x|^2$ (that is for the renormalized (PME))  Lee and Vazquez \cite{LV} showed that the interface becomes convex in finite time.  It is unknown whether such results hold for general convex potentials: we shall investigate this in an upcoming work.
\end{itemize}
\end{remmark}

\section{Viscosity Solution}

In this section we introduce the appropriate notion of viscosity solution for (PME--D) and show that it is equivalent to the usual notion of weak solution.  Our definition descends from those in \cite{CV} and \cite{Kim}.  For more details we also refer the reader to the definitions, discussions and results in \cite{BV}.  

\subsection{Definition and Basic Properties}

Let $Q:= \R^n\times (0,\infty)$.  For a nonnegative function $u(x,t)$ in $Q$, we define the {\it positive phase}
$$
\Omega(u) =\{u>0\},\quad \Omega_t(u):=\{x: u(x,t)>0\}
$$ 
and the {\it free boundary } 
$$
\Gamma(u)=\partial\Omega(u), \quad \Gamma_t(u):=\partial\Omega_t(u).
$$
As in \cite{CV}, to describe the free boundary behavior using comparison arguments we need an appropriate class of test functions to handle the degeneracy of (PME--D).

Let $\Sigma$ be a smooth, cylinder--like domain in $\R^n\times [0,\infty)$, i.e., 
\begin{equation}\label{domain}
\Sigma = \bigcup_{t_1\leq t\leq t_2} \Sigma(t)\times\{t\}, \hbox{ where } \Sigma(t)\hbox{ is a smooth domain in } \R^n.
\end{equation}

\begin{defn}\label{free_bound_soln}
A nonnegative function $u \in C^{2,1}(\overline{\{u>0\}}\cap\Sigma)$ is a {\it classical free boundary subsolution} in $\Sigma$ if 
\begin{itemize}
\item[$\circ$]$u$ satisfies (PME--D) with $\leq$ replacing $=$ in the classical sense in $\{ u > 0\}\cap\Sigma$; 
\item[$\circ$]$|Du|>0$  on $\Gamma(u)\cap\Sigma$ with outward normal velocity
$$
V \leq |\nabla u| + \nabla \Phi \cdot \frac{\nabla u}{|\nabla u|}\hbox{ ~~~on } \Gamma(u),
$$
or, equivalently, 
$$
 u_t \leq |\nabla u|^2 + \nabla \Phi \cdot \nabla u\hbox{ ~~~on } \Gamma(u).
$$
\end{itemize} 
 \medskip
 
We define a {\it classical free boundary supersolution} by replacing $\leq$ with $\geq$.  

Finally, $u$ is a \emph{classical free boundary solution} if it is both a sub-- and supersolution.  
\end{defn}

Before proceeding further it is convenient to introduce some auxiliary definitions.  

\begin{defn}
Let $\varphi$ be a continuous, nonnegative function.  Now if $\psi$ is another such function,  then we say that $\varphi$ {\it touches} $\psi$ {\it from above} at $(x_0, t_0)$ in $\Sigma$ if $\varphi - \psi$ has a local minimum zero at $(x_0,t_0)$ in $\Sigma \cap \{ t \leq t_0\}$.  We have a similar definition for $\varphi$ {\it touching } $\psi$ {\it from below}.
\end{defn}

\begin{defn}[Strictly separated data]
For two nonnegative functions $u, v: \R^n\to \R$,  we write $u_0 \prec v_0$ if the following holds:
$\mbox{supp}(u_0)$ is compact and $\mbox{supp}(u_0) \subset \mbox{Int(supp(}v_0))$ and inside $\mbox{supp}(u_0)$, $u_0(x) < v_0(x)$.   
\end{defn}

  We note that e.g., due to the maximum principle, a classical free boundary subsolution that lies below a classical free--boundary supersolution at time $t_1 \geq 0$ cannot cross the supersolution from below at a later time $t_2 > t_1$.  This observation leads to a notion of viscosity solution which takes into account the free boundary: 
\begin{defn}\label{viscosity}
Let $u$ be a continuous, nonnegative function in $Q$.
\begin{itemize}
\item[$\circ$] $u$ is a \emph{viscosity subsolution} of (PME--D) if, for any given smooth domain $\Sigma$ given in (2.1), for every $\varphi \in C^{2, 1}(\Sigma)$ that touches $u$ from above at the point $(x_0, t_0)$, we have 
\begin{equation}\label{visc_defn_sub}\varphi_t \leq (m-1) \varphi \Delta \varphi + |\nabla \varphi|^2 + \nabla \varphi \cdot \nabla \Phi + (m-1)\varphi \Delta \Phi.\end{equation}

\item[$\circ$] $u$ is a \emph{viscosity supersolution} of (PME--D) if, for any given smooth domain $\Sigma$ as given in (2.1),\\
(i) for every $\varphi \in C^{2, 1}(\Sigma)$ that touches $u$ from below at the point $(x_0, t_0) \in \Omega(u)\cap \Sigma$, we have 
\begin{equation}\label{visc_defn_sup}\varphi_t \geq (m-1) \varphi \Delta \varphi + |\nabla \varphi|^2 + \nabla \varphi \cdot \nabla \Phi + (m-1)\varphi \Delta \Phi.\end{equation}

(ii) for every classical free--boundary subsolution $\varphi$ in $\Sigma$, the following is true:
If $ \varphi \prec u$ on the parabolic boundary of $\Sigma$, then $\varphi \leq u$ in $\Sigma$.  That is, every classical free--boundary subsolution that lies below $u$ at a time $t_1 \geq 0$ cannot cross $u$ at a later time $t_2 > t_1$.\\

\item[$\circ$] $u$ is a {\it viscosity solution} of (PME--D) with initial data $u_0$ if $u$ is both a super-- and subsolution and $u$ uniformly converges to $u_0$ as $t\to 0$.

\end{itemize}
\end{defn}  

\begin{remark} In general one can define viscosity sub-- and supersolutions respectively  as upper-- and lower semicontinuous functions. Such a definition turns out to be useful when one cannot verify continuity of solutions obtained via various limits. This problem does not arise in our investigation here thanks to \cite{BH}, and therefore our definition assumes continuity of solutions.  
\end{remark}

It is fairly straightforward to verify that a classical free boundary sub-- (super)solution is also a viscosity sub-- (super)solution.

\begin{lemma}\label{classical_visc}
If $w$ is a classical free boundary sub-- (super) solution to (PME--D), then $w$ is also a viscosity sub-- (super) solution.
\end{lemma}
\begin{proof}
We will be brief: The subsolution case presents no difficulty since if contact with some $\varphi \in C^{2, 1}(\Sigma)$ occurs in $\Omega(w)$ then we use the fact that $w$ is classical there, whereas no contact can occur on the free boundary unless $|\nabla w| = 0$, in which case the differential inequality is satisfied since then $\varphi = |\nabla \varphi| = 0$ and $\varphi_t \leq 0$.  

If $w$ is a classical free boundary supersolution, then (i) in Definition \ref{viscosity} follows as before.  To see (ii), let us note that if $\varphi$ is a classical free boundary subsolution which crosses $w$, then since the free boundary is $C^2$, Hopf's Lemma implies that at the touching point $|\nabla \varphi| < |\nabla w|$ (see e.g., \cite{max_principle}).  On the other hand, since $\varphi$ started below $w$, at the touching point we must have $v_n(\varphi) \geq v_n(w)$, which leads to a contradiction since it is also the case that we have $\frac{\nabla w}{|\nabla w|} = \frac{\nabla \varphi}{|\nabla \varphi|}$.

\end{proof}

Next we have the following stability result.

\begin{lemma}\label{stability}
Let $u^\e$ be a smooth solution of (PME--D) with initial data $u_0+\e$ and let $u$ be its uniform limit.
Then $u$ is a viscosity solution of (PME--D) with initial data $u_0$. 
\end{lemma}
\begin{proof}
Let $\Sigma$ be as given in (\ref{domain}) and let  $\varphi\in C^{2,1}(\Sigma)$. 

\medskip

1. Let us first show that $u$ is a subsolution.  First suppose that $\varphi$ touches $u$ from above at the point $(x_0, t_0)$.  We may assume that $u - \varphi$ has a strict maximum at $(x_0, t_0)$ in $\overline \Omega(u) \cap \Sigma \cap \{t \leq t_0\}$ by replacing $\varphi$ by 
\[ \tilde \varphi(x, t): = \varphi(x, t) + \sigma( (x-x_0)^4 - (t- t_0)^2), ~~~\sigma > 0\]
if necessary.  By uniform convergence there exists a sequence $(x_\varepsilon, t_\varepsilon)$ converging to $(x_0, t_0)$ such that $u^\varepsilon - \varphi$ has a local maximum at $(x_\varepsilon, t_\varepsilon)$.  Now if we we let 
\[ \tilde \varphi(x, t) := \varphi(x, t) - \varphi(x_\varepsilon, t_\varepsilon) + u^\varepsilon(x_\varepsilon, t_\varepsilon)\]
Then $u^\varepsilon - \tilde \varphi$ has a local maximum at $(x_\varepsilon, t_\varepsilon)$ with $(u^\varepsilon-\tilde \varphi)(x_\varepsilon, t_\varepsilon) = 0$.  We can now conclude by taking the limit of the viscosity subsolution property of $u^\varepsilon$. 

\medskip      
2.  Next we show that $u$ is a supersolution.  Let $\varphi$ be a classical free--boundary subsolution such that $\varphi(x, t_1) \prec u(x, t_1)$.  Since the $u_\varepsilon$'s are strictly ordered, $u < u_\varepsilon$ and hence $\varphi(x, t_1) \prec u(x, t_1) < u_\varepsilon(x, t_1)$.  Now suppose $\varphi$ touches $u_\varepsilon$ at some point $(x_2, t_2)$, then $\varphi(x_2, t_2) > 0$ since $u_\varepsilon$ is positive, so by continuity, there is a parabolic neighborhood of $(x_2, t_2)$ in which both functions are classical and positive.  By the Strong Maximum Principle, the touching cannot have occurred at $(x_2, t_2)$, a contradiction.  We conclude that $\varphi < u_\varepsilon$ so that in the limit $\varphi \leq u_\varepsilon$.

\end{proof}
An immediate consequence of the above lemma is that weak solutions are viscosity solutions (see Corollary \ref{weak}). We shall introduce the precise notion of weak solutions in the next subsection, and summarize some results from \cite{BH}.

\subsection{Weak Solutions}
To be consistent with the setup in both \cite{BH} and \cite{jose_et_al}, let us return to the density variable and consider the solution of \eqref{main_eq_density} in a bounded domain $\Omega$ with Neumann boundary condition: 
$$
\left\{\begin{array}{lll}
\rho_t = \Delta \rho^m + \nabla \cdot (\rho \nabla \Phi) &\hbox{in }& \Omega \times \mathbf R^+,\\ \\
 \partial \rho^m/\partial \nu + \rho \left(\partial \Phi/\partial \nu\right) = 0 &\mbox{ on }& \partial \Omega \times \mathbf R^+,\\ \\
 \rho(x, 0) = \rho_0(x) &\hbox{ in }&\Omega.
\end{array}\right. \leqno (N)
$$
We will see shortly that we need not worry about the fact that we are on a bounded domain, but for now we will let $Q = \Omega \times \mathbf R^+$ and $Q_t= \Omega \times (0, t]$.  As in \cite{BH}, we make the following definition:
\begin{defn}
We say $\rho: [0, \infty) \rightarrow L^1(\Omega)$ is a \emph{weak solution} of (N) if\\
 
(i) $\rho \in C([0, t]; L^1(\Omega)) \cap L^\infty(Q_t) \mbox{ for all } t \in (0, \infty)$;\\

(ii) for all test functions $\varphi \in C^{2,1}(\overline Q)$ such that $\varphi \geq 0$ in $Q$ and $\partial \varphi/\partial \nu = 0$ on $\partial \Omega \times \mathbf R^+$, we have 
$$
\int_\Omega \rho(t) \varphi(t) = \int_\Omega \rho(0) \varphi(0) + \int \int_{Q_t} (\rho \varphi_t + \rho^m \Delta \varphi - \rho \nabla \Phi \cdot \nabla \varphi)
$$
We also define a weak subsolution (respectively supersolution) by (i) and (ii) with equality replaced by $\leq$ (respectively $\geq$).
\end{defn}

From \cite{BH} we have existence, regularity, uniqueness and comparison principle for weak solutions:
\begin{thm}[From \cite{BH}]\label{weak_soln_results}
Under the assumption that  $\Phi$ is $C^2$ in $\bar{\Omega}$: 
\begin{itemize}
\item[(a)] the problem (N) has a unique solution;\\
\item[(b)] the solution is uniformly bounded in $Q$ and is continuous in any set $\overline \Omega \times [0, T]$;\\
\item[(c)] suppose $\underline \rho(t)$ is a subsolution and $\overline \rho(t)$ is a supersolution, then if $\underline \rho_0 \leq \overline \rho_0$ in $\Omega$, then $\underline \rho(t) \leq \overline \rho(t)$ in $\Omega$ for $t \geq 0$. 
\end{itemize}  
\end{thm}

The existence of solutions is obtained as the uniform limit of solutions to uniformly parabolic problems (equicontinuity is obtained from \cite{DiB}).  For our purposes, a very simple approximation basically suffices and we summarize the relevant result in the following:

\begin{lemma}[From \cite{BH}]\label{equicontinuity}
Let $u^\varepsilon$ be a solution of (N) with initial data $u_0^\varepsilon = u_0 + \varepsilon$, then $u^\varepsilon$ is equicontinuous and there exists a subsequence which uniformly converge to $u$ which is the unique weak solution to (N) with initial data $u_0$.  
\end{lemma}

While \emph{a priori} our viscosity solution is defined in all of $\mathbb R^n$, since (formally at least) solutions of (PME--D) should have finite speed of propagation, the boundary conditions should be inconsequential if we take $\Omega$ sufficiently large. (Later we will also establish finite propagation for viscosity solutions -- see Corollary \ref{finite_propagation}.)  Control on the speed of expansion of the support can be done via comparison with any (weak) supersolution.  In particular, when $\Phi$ is convex, we can use the stationary profiles of the form $\Psi(x) = (C-\Phi)_+$ with sufficiently large $C$ as a supersolution (see Theorem~\ref{equilibrium_solution}).  

\begin{remark}
Alternatively (and perhaps this is a cleaner line of reasoning), we can directly use the result of \cite{JGJ} on existence and uniqueness of solutions in all of $\mathbb R^n$, which implies in particular that the results of \cite{jose_et_al} also apply in that setting.
\end{remark}

Combining Lemma \ref{stability} with Lemma \ref{equicontinuity} and the uniqueness statement in Theorem \ref{weak_soln_results}, we obtain:
\begin{cor}\label{weak}
Any weak solution is also a viscosity solution.  
\end{cor}

We will eventually establish uniqueness of viscosity solutions via maximum--principle type arguments, which culminates in the identification of the two notions of solution.

\subsection{Construction of test functions }\label{barriers}

In this subsection we collect some test functions, i.e., (classical free boundary) sub-- (super) solutions, to (PME--D) which will be useful for comparison purposes.  In the first couple of lemmas (Lemmas \ref{parabolic_super_barrier} and \ref{parabolic_sub_barrier}) the idea is to control the $\Phi$ dependence via Taylor expansion in a small neighborhood of a point, so that we can appropriately perturb the test functions for (PME) constructed in \cite{BV} and \cite{CV} for our purposes.  

Let $\alpha > 0$ be small.  We will first explain how locally the (PME--D) can be viewed as a perturbation of (PME). The starting point is to observe that if $u$ is a solution of (PME--D) in the parabolic cylinder
\[ Q_\alpha (x_0, t_0) = B(x_0, \alpha) \times [t_0, t_0 + \alpha]. \]
By Taylor's theorem, $u$ satisfies an equation
\[u_t = (m-1)u \Delta u + |\nabla u|^2 + \vec b \cdot \nabla u +  (m-1)u \Delta \Phi + O(\alpha)[D^2\Phi] \nabla u + O(\alpha^2),\]
where 
\[\vec b = \nabla \Phi(x_0, t_0)\]
and $[D^2\Phi]$ denotes the relevant term in Taylor's expansion for $\nabla \Phi$.

On the other hand, if we use hyperbolic scaling and define 
\[v(x, t) = \alpha^{-1}u(\alpha (x-x_0), \alpha (t-t_0)),\]
then $v$ satisfies
\[ v_t = (m-1) v\Delta v + |\nabla v|^2 + \nabla v \cdot \nabla \Phi + \alpha(m-1) v \Delta \Phi\]
in the parabolic cylinder
\[ Q_1(0, 0) = B(0, 1) \times [0, 1].\]
Combining this with the Taylor expansion, we see that $v$ satisfies 
\[\begin{split} v_t &= (m-1) v\Delta v + |\nabla v|^2 + \vec b \cdot \nabla u + \alpha (m-1) v \Delta \Phi + O(\alpha) [D^2 \Phi] \nabla v + O(\alpha^2)\\
&:= (m-1) v\Delta v + |\nabla v|^2 + \vec b \cdot \nabla u + R(x_0, v).\end{split}\]
From Taylor's theorem, we see that 
\begin{equation}\label{error_est} |R(x_0, v)| \lesssim \alpha\|\Phi\|_{C^2} (|v| + |\nabla v|).\end{equation}
Finally, we define 
\[ w(x, t) = v(x - \vec b t, t).\]
Then $w$ satisfies 
\begin{equation}\label{final_w} w_t = (m-1) w \Delta w + |\nabla w|^2 + R(x_0, w), \end{equation}
which is indeed the (PME) with an $O(\alpha)$ perturbation in the unit scale parabolic cylinder
\[ Q_1(\vec b t, 0) = B(\vec b t, 1) \times [0, 1].\]

%The starting point is to observe that if we consider (PME--D) in some small cylinder $Q_\alpha := B_\alpha(x_0) \times [t_0- \alpha, t_0+\alpha]$ and define $v_1(x, t) = \alpha^{-1}u(\alpha (x-x_0), \alpha (t-t_0))$, then $v_1$ satisfies, in the unit cylinder $B_1(0)\times [-1,1]$, an equation of the type   
%\[(v_1)_t = (m-1) v_1\Delta v_1 + |\nabla v_1|^2 + \vec{b} \cdot \nabla v_1+ O(\alpha) D v_1 + \alpha(m-1)v_1\Delta\Phi,\]
%where $\vec{b} = \nabla \Phi(x_0, t_0)$. The size of the last two terms  depends on the $C^2$-norm of $\Phi$ in $Q_\alpha$.  
%
%Next we take 
%\begin{equation}\label{exp_transform} v(x, t) = v_1(x - \vec{b}t, t)\end{equation}
%Then $v$ satisfies
%\begin{equation}\label{eq_v} v_t = (m-1) v\Delta v + |\nabla v|^2 + O(\alpha)\nabla v +O(\alpha) v.
%\end{equation}

To construct test functions for (PME--D) from perturbations of solutions of (PME) we shall have to reverse the order of operations:  We will construct super-- (sub)solutions of \eqref{final_w} by perturbing corresponding super-- (sub)solutions of the (PME) at the unit scale and then we will use hyperbolic scaling to arrive at a rescaled (and translated) local super-- (sub)solution for the (PME--D).  Let us note for future reference that 
\[ u(x, t) = \alpha w \left(\alpha^{-1}(x + \vec b t), \alpha^{-1}t\right).\]

The next proposition gives one way to perform the necessary perturbation on solutions of (PME) to arrive at \eqref{final_w}.  

\begin{prop}\label{prop:perturb}
Let $u(x,t)$ be a viscosity subsolution of  (PME) in $B_{1+\alpha}(0)\times [-1,1]$. Then for $0<\alpha<1$,
 $$
 u_1(x,t) := e^{-\alpha t} \sup_{y\in B_{\alpha-\alpha t} (x)}  u(y,t)
 $$ is a subsolution of
$$
(u_1)_t =(m-1) u_1\Delta u_1+ |\nabla u_1|^2 - \alpha|\nabla u_1| - \alpha u_1\leqno \mbox{(PME--sub)}
$$
in $B_1(0)\times [-1,1]$.
Similarly, if $u(x,t)$ is a viscosity supersolution of (PME) in $B_1(0)\times [-1,1]$, then for $0<\alpha<1$,

$$
u_2(x,t):= e^{\alpha t} \inf_{y\in B_{\alpha-\alpha t} (x) }  u(y,t)
$$ is a supersolution of 
$$
(u_2)_t = (m-1)u_2\Delta u_2 + |\nabla u_2|^2 +\alpha|\nabla u_2|+\alpha u_2 \leqno \mbox{(PME--super)}
$$
in $B_1(0)\times [-1,1]$.
\end{prop}

\begin{proof}
We only show the supersolution part. Let $u_2$ be as given above.
Suppose then that $\varphi$ is classical and touches $u_2$ from below at some point $(x_0, t_0)$.  We first note that there exists $x_1\in \overline B_{\alpha - \alpha t} (x_0)$ such that $u_2(x_0, t_0) = e^{\alpha t_1}u(x_1, t_0)$.

Next for any unit vector $\hat{b}$, let us consider 
\[\tilde \varphi(x, t) = e^{-\alpha t}\varphi(x - (x_1 - x_0) - \alpha \hat{b} (t - t_0),  t).\]
Then we note that 1) $\tilde \varphi(x_1, t_0) = e^{-\alpha t_0}\varphi(x_0, t_0)$ and so $(u- \tilde \varphi)(x_1, t_0) = 0$ and 2) by the definition of $u_2$ as an infimum and by continuity of $\varphi$, in a small parabolic neighborhood of $(x_0, t_0)$, it is the case that $u - \tilde \varphi \geq u_2 - \tilde \varphi \geq 0$; we therefore conclude that $\tilde \varphi$ touches $u$ from below at $(x_1, t_0)$ and so we have, 
$$
\begin{array}{lll}
[\varphi \Delta \varphi + |\nabla \varphi|^2](x_0, t_0) &=&e^{\alpha t_0} [\tilde \varphi \Delta \tilde \varphi + |\nabla \tilde \varphi|^2](x_1, t_0) \\ \\
&\leq& e^{\alpha t_0}\tilde \varphi_t(x_1, t_0) \\ \\
&=&[\varphi_t -\alpha\varphi-\alpha \hat b \cdot \nabla \varphi](x_0, t_0).
\end{array}
$$  
Now the desired inequality is achieved by setting $\hat b = \frac{\nabla \varphi}{|\nabla \varphi|}(x_0, t_0)$.

Indeed the above calculation shows that if $u$ is a viscosity supersolution of (PME), then $u_2$ should be a viscosity supersolution of (PME--super): If  a classical free boundary subsolution $\varphi$ of (PME--super) crosses $u_2$ from below, then the corresponding $\tilde \varphi$ is a subsolution of (PME) and crosses $u$, yielding a contradiction (there is no distinction between the interior and boundary cases).   
\end{proof}

We will use the spherically symmetric supersolutions for (PME) from \cite{CV}:

\begin{lemma}\label{sphere_symm}
Consider the function 
\[ H(x, t; A, \omega) = A(|x| + \omega t - B)_+\]
with $R/2 < B < R$.  Then $u$ is a classical free boundary supersolution of (PME) in the domain $\{|x|\leq R\}\times [\omega^{-1}(B-R),0]$ if 
\[ \frac{\omega}{A} > 1 + 2(m-1)(n-1)\frac{R-B}{R}.\]
\end{lemma}

Proposition~\ref{prop:perturb} and Lemma~\ref{sphere_symm} yield the following:

\begin{cor}\label{parabolic_super_barrier}
Let us fix $x_0\in\R^n$ and let $H$ be given as in Lemma \ref{sphere_symm}.
Then the inf convolution of $H$, given as 
$$
\underline H(x, t;\alpha) = e^{\alpha t}\inf_{y\in B_{\alpha - \alpha t}(x)} H(y, t) 
$$
is a classical (free boundary) supersolution of (PME--super). Consequently,  there exists $C=C_0$ which only depends on the $C^2$-norm of $\Phi$ in $B_1(x_0)$ such that 
$$
\tilde{H}(x, t) := \alpha \underline H(\alpha^{-1}(x-x_0+\vec{b}(t-t_0)),\alpha^{-1}(t-t_0);C\alpha)
$$ is a classical (free boundary) supersolution of (PME--D) in  $Q_\alpha:= B_\alpha(x_0) \times [t_0- \alpha, t_0]$.
\end{cor}
 \begin{proof}
 By Lemma \ref{classical_visc} and Proposition \ref{prop:perturb}, $\underline H$ is a viscosity supersolution of (PME--super) and hence adjusting $\alpha$ by multiplying by a suitable constant to take into account the $C^2$--norm of $\Phi$, we obtain a viscosity supersolution of \eqref{final_w} from which we obtain a supersolution of (PME--D) in $Q_\alpha$.  
 
 To finish it is sufficient to show that $\tilde H$ has the required regularity to be a classical free boundary supersolution.  For this, note that for simplicity we have only taken the supremum over space and the reader can readily check that in this case the infimum for $\underline H(x,t;\alpha)$ is achieved at the point $y$ which minimizes $|y|$ subject to the constraint that $|y - x| = \alpha - \alpha t$, and thus an explicit expression is possible for $\underline H$. 
 \end{proof}

%\begin{remark}\label{remark_boundary_vel}
%In fact due to the explicit form of $H$ it follows that the free boundary velocity of $\tilde{H}$ is given by 
%$$
%V=\omega +\vec{b}\cdot \frac{\nabla H}{|\nabla H|}+C\alpha.
%$$
%\end{remark}

By comparison with these supersolutions, we immediately obtain 
\begin{cor}\label{finite_propagation}
Any viscosity solution has finite speed of propagation and is bounded in a big ball in any local time interval.  Further, if $\Phi$ is convex, then via comparison with the stationary solutions of the form given in Theorem \ref{equilibrium_solution}, the above holds globally in time.
\end{cor}

For subsolutions we will make use of the Barenblatt profiles (see e.g., \cite{CV} and \cite{V}).

\begin{lemma}\label{barenblatt}[Barenblatt]
Let $B(x, t; \tau, C)$ be the family of functions 
\[ B(x, t; \tau, C) = \frac{(C(t+\tau)^{2\lambda} - K|x|^2)_+}{(t+\tau)}\]
with constants $\lambda, K, C, \tau > 0$ such that 
\[ \lambda = ((m-1)d + 2)^{-1}, ~~~2K = \lambda.\]
Then $B(x, t; \tau, C)$ is a classical (free boundary) solution of (PME).
\end{lemma}
Using Proposition~\ref{prop:perturb} (see also the proof of Corollary \ref{parabolic_super_barrier}) once again, we obtain the following:
\begin{cor}\label{parabolic_sub_barrier}
Let us fix some $(x_0, t_0)$ and let $B$ be a Barenblatt function.  
Then there exists a constant $C$ which only depends on the $C^2$-norm of $\Phi$ in $B_1(x_0)$ such that 
 $$
 \tilde B(x, t) = \alpha e^{-C\alpha (t-t_0)}\sup_{y\in B_{C\alpha-C(t-t_0)}(x)}B(\alpha^{-1}(y-x_0+\alpha \vec{b}(t-t_0)), \alpha^{-1}(t-t_0))
 $$ is a classical (free boundary) subsolution of (PME--D) in $Q_\alpha:= B_\alpha(x_0) \times [t_0- \alpha, t_0]$.
\end{cor}

\begin{remark}\label{perturbed_velocity}
The reason for taking the hyperbolic scaling is because we will have occasion to require rather fine control on the boundary velocity (see Lemma \ref{nontangential}) and this is the scaling which preserves the velocity -- in contrast to the parabolic scaling, which dramatically reduces the effect of the drift $\Phi$ in the bulk (the positivity set), but unfortunately at the cost of severely disrupting the boundary velocity.
\end{remark}

To establish the Comparison Principle, we will need the following weak analogue of (ii) in the definition of viscosity supersolutions for subsolutions, the proof of which utilizes an approximation lemma from \cite{BV}.

\begin{lemma}\label{no_cross_from_above}
Let $u$ be a viscosity subsolution of $(PME$--$D)$, and let $\varphi$ be a classical free boundary supersolution from Lemma \ref{parabolic_super_barrier} which lies above $u$ at some time $t_0$.  Then $\varphi$ cannot cross $u$ from above at a later time $t > t_0$.
\end{lemma}
\begin{proof}
Let $\varphi$ and $u$ be as described in the statement, and suppose that $\varphi$ touches $u$ from above at some point $(x_0, t_0)$.  From Lemma \ref{parabolic_super_barrier}, we have that $\varphi$ is given as the inf convolution of some spherical traveling waves from Lemma \ref{sphere_symm}, which we denote by $\psi$.  Further, let us suppose the infimum is achieve at $(x_1, t_1)$ so that 1) $\varphi(x_0, t_0) = \psi(x_1, t_1)$ and 2) by the definition of $\varphi$ as an inf convolution, the translated function 
$$
\tilde \psi(x, t) = \psi(x + (x_1 - x_0), t + (t_1 - t_0))
$$ also touches $u$ from above at the point $(x_0, t_0)$.  From Lemma 4.4 in \cite{BV}, we know that $\psi$ can be given as the monotone limit of classical positive supersolutions, and hence the same is true of $\tilde \psi$: I.e., there exsits $\psi_\varepsilon \searrow \tilde \psi$, with $\psi_\varepsilon > 0$ classical.  But since $u$ cannot touch $\psi_\varepsilon$ by the Strong Maximum Principle, we obtain in the limit that $u \leq \tilde \psi$, which is a contradiction.
\end{proof}

\subsection{Comparison Principle and Identification with Weak Solution}

Here the outline of the proof closely follows that of the corresponding result for (PME) (Theorem 2.1 in \cite{CV}): We will give an abridged version of the proof, pointing out main steps and modifications for our problem.  Here we let $Q = \mathbb R^n \times \mathbb R^+$.

\begin{thm}\label{thm:cp} [Comparison Principle]
If $u$ is a viscosity subsolution and $v$ is a viscosity supersolution in the sense of Definition \ref{viscosity} with strictly separated initial data, $u_0 \prec v_0$, then $u(x, t) \leq v(x, t)$ for every $(x, t) \in Q$. 
\end{thm}
\begin{proof}
1) [Sup and Inf Functions]  For given $\delta > 0$ and $r > 0$ small with $r \gg \delta$ we introduce the regularized functions 
\[ W(x, t) = \inf_{\overline B_{r - \delta t}(x, t)} v(y, \tau)\]
and 
\[ Z(x, t) = \sup_{\overline B_r(x, t)} u(y, \tau)\]
First  note that $W$ and $Z$ preserve properties of $v$ and $u$: 
\begin{itemize}
\item [$\circ$] \emph{$W$ is a supersolution and $Z$ is a subsolution;}
\item [$\circ$] \emph{$Z(\cdot, r) \prec W(\cdot, r)$ for $r$ sufficiently small.}
\end{itemize}  
For a proof of the first item see \cite{BV}, Lemma 7.1 or the proof of Lemma \ref{parabolic_super_barrier}.  The proof of the second item relies on the fact that 
\begin{itemize}
\item [$\circ$] \emph{The support of viscosity subsolutions and supersolutions evolve in a continuous way.  Here continuity is understood as continuity in the Hausdorff distance (in time) of the positivity set.}
\end{itemize}
The proof of this can be done by comparison with the supersolutions (respectively subsolutions) constructed in Corollary \ref{parabolic_super_barrier} (respectively Corollary \ref{parabolic_sub_barrier}).  We omit the details since with replacement of barriers it is no different from the proof of Proposition 6.2 in \cite{BV}.    

Thus if we can prove that $W$ stays above $Z$ for all choices of $r$ and $\delta$ (sufficiently small), then we may take $\delta \rightarrow 0$ and then $r \rightarrow 0$ to recover the conclusion for $u$ and $v$.  First let us note that due to the Strong Maximum Principle, $W$ cannot touch $Z$ from above, and therefore we are reduced to the analysis of a first contact point of $W$ and $Z$ at some $P_0 = (A, t_0)$. 

The key usefulness of $Z$ and $W$ lies in the fact that they enjoy interior/exterior ball properties: 
\begin{itemize}
\item [$\circ$] \emph{The positivity set of $Z$ has the interior ball property with radius $r$ at every point of its boundary \emph{and} at the points of the boundary of the support of $u$ where these balls are centered we have an exterior ball;}
\item [$\circ$] \emph{The positivity set of $W$ has the exterior ball property with radius less than $r - \delta t$ (since in this case we really have an exterior ellipsoid in space--time) and at the points of the boundary of the support of $v$ where these balls are centered we have an interior ball.}
\end{itemize}
For detailed proofs of these statements we again refer the reader to \cite{BV}. 

2) [The Contact Point] The first contact point $P_0 = (x_0, t_0)$ is located at the free boundary of both functions.  Therefore by the definitions of $Z$ and $W$, there is a point $P_1 = (x_1, t_1)$ on the free boundary of $u$ located at distance $r$ from $P_0$ and there is another point $P_2 = (x_2, t_2)$ on the free boundary of $v$ at distance $r_0 = r - \delta t_0$ from $P_0$.  Let us also denote by $H_Z$ (respectively $H_W$) the tangent hyperplane to the free boundary of $Z$ (respectively $W$) at $P_0$.  (see Figure \ref{fig_contact_point})
\begin{figure}
\center{\includegraphics[height=3.5in, width=3.5in]{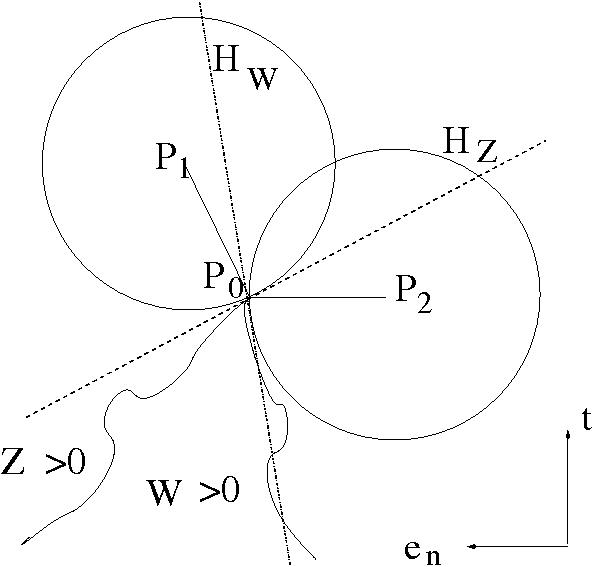}} \caption{The geometry at the contact point}
\label{fig_contact_point}
\end{figure}

\begin{lemma}\label{speed_well_defined}
Neither $H_Z$ nor $H_W$ is horizontal. In particular, one can denote the space-time normal vector to $H_Z$, in the direction of $P_1-P_0$, as $(e_n,m)\in\R^n\times \R$ where $|e_n|=1$ and $-\infty < m<\infty$.
\end{lemma}
\begin{proof}
It is enough to show that $t_1>t_0-r$ (i.e., $\Gamma(Z)$ cannot propagate with infinite speed) and $t_2 <t_0+r $ (i.e., $\Gamma(W)$ cannot propagate with negative infinite speed).  The desired conclusion then follows by the ordering of $Z$ and $W$.

We first show that $t_2<t_0+r$.  Otherwise $H_W$ is horizontal and after translation we have $P_0 = (0, -r)$ and $P_2 = (0, 0)$.  Moreover $\Omega(v)$ has an interior ball at $P_2$ with horizontal tangency with radius $0 < r^\prime < r$.  Now in any parabolic cylinder 
$$
C_{\eta} = \{ (x, t): |x| \leq \eta, -\eta^2 \leq t \leq 0\}
$$
 with bottom edge contained in the interior ball (which can be achieved by taking $\eta \leq r^\prime$), we have by continuity that $v \geq M > 0$ on that edge. (see Figure \ref{fig_horizontal_tangent})
\begin{figure}
\center{\includegraphics[height=3.5in]{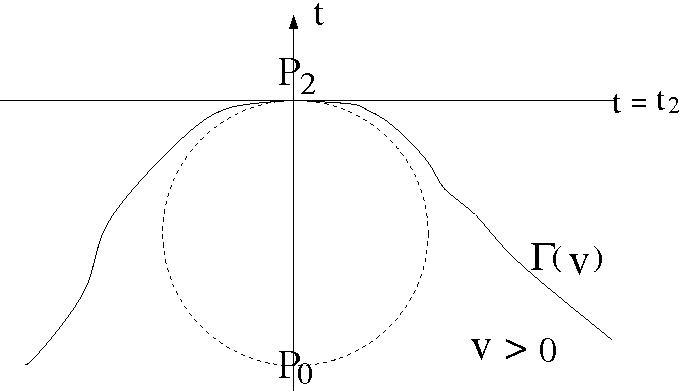}}
\caption{Infinite, negative speed}\label{fig_horizontal_tangent}
\end{figure}
 
On the lateral boundary of $C_\eta$ it may be the case that $v = 0$, so we will have to compare with a subsolution with support strictly contained in $-\eta < |x| < \eta$ at time $t=-\eta^2$ and still contains $0$ in its support at time $t = 0$, which rigorously implies that $v$ cannot contract sufficiently fast for $(0, 0)$ to be a free boundary point.  The necessary subsolution can be constructed as the one in Lemma \ref{parabolic_sub_barrier}, adjusted for the parabolic scaling.

The case $t_1>t_0-r$ follows \emph{mutatis mutantis} from the arguments in \cite{CV} and \cite{BV} (see \cite{CV}, Lemma 4.2 or \cite{BV}, Lemma 8.2) with the barrier in Lemma \ref{parabolic_super_barrier} replacing the barriers used therein.  
\end{proof}

3) [Non--tangential Estimate] 
The next lemma states that the normal velocity $V$ of $\Gamma(Z)$ at $(x_0,t_0)$ satisfies, in the viscosity sense,
$$
V_{(x_0,t_0)} \leq \left(|\nabla Z| +\nabla \Phi(x_0)\cdot\frac{\nabla Z}{|\nabla Z|}\right)(x_0,t_0).
$$

\begin{lemma}\label{nontangential}
Let $x_n:= x\cdot e_n$, and consider a non-tangential cone 
\[K:=\{x: x_n \geq k|x| \mbox{ with $k>0$}\}.\]  Then we have 
$$
\lim_{x\in K,~x \to 0}\dfrac{Z(x_0+x, t_0)}{x_n} \geq m- \nabla \Phi(x_0)\cdot e_n.
$$
\end{lemma}
\begin{proof}
The argument is parallel to the proof of Lemma 4.3 in \cite{CV}; the only difference for us is taking into account the change of reference frame introduced by the drift given by $\Phi$. 
This is ensured by the local nature of the construction of our barrier in Lemma \ref{parabolic_super_barrier}, which replaces the corresponding barriers used in \cite{CV}.
 
\end{proof}

4) [Conclusion] Due to Lemma \ref{nontangential}, we may place a small subsolution $\varphi$ from Lemma \ref{parabolic_sub_barrier} below $Z$ at $P_0$ with speed close to $m$ (again see Remark \ref{perturbed_velocity}, which assures us that our subsolutions are constructed so that this is possible) such that it crosses anything with speed $m^\prime < m$.  Since $\varphi$ is also below $W$ and hence $v$ (after a small translation), $v$ must expand by at least $m^\prime$, but then $\Gamma(W)$ has speed $m^\prime + \delta > m$ at $P_0$, yielding a contradiction to the fact that $Z$ touched $W$ from \emph{below} at $P_0$.
\end{proof}

We can now establish uniqueness of viscosity solutions:

\begin{thm}\label{exists_unique}
The problem (PME--D) admits a unique solution in the class of viscosity solutions as defined in Definition \ref{viscosity} for continuous and nonnegative initial data.  This solution coincides with the continuous weak solution. 
\end{thm}

\begin{proof}  
The existence of a continuous weak solution can be provided as the uniform limit of classical solutions with initial data $u_{0, \varepsilon} = u_0 + \varepsilon$, and by Lemma \ref{stability}, such a limit, which we will denote by $U$, is also a continuous viscosity solution.  Further, by comparison with $u_{0, \varepsilon}$ and taking a limit, it is clear that  such a limit $U$ is also a maximal viscosity solution.  

Uniqueness would follow if we can show that any other viscosity solution $u$ also cannot be smaller than $U$.
For this purpose, consider $u_n(x,t)$ with initial data $u_n(x,0):= (u_0-\frac{1}{n})_+$.    Now consider positive $u_n^{\e_n}$ such that $|u_n^{\e_n}-u_n| <\frac{1}{n}$ in $\R^n\times [0,T]$. It follows from Lemma \ref{equicontinuity} that $u_n^{\e_n}$  uniformly converges to $U_2(x,t)$, which is then a continuous weak solution of (PME--D).  Therefore, by uniqueness of weak solutions, $U_2$ is equal to $U$.  On the other hand by Theorem~\ref{thm:cp} $u_n \prec u$ and thus $U=U_2 \leq u$. Hence we conclude.
\end{proof}

Using Theorem \ref{thm:cp} and Theorem \ref{exists_unique}, we can in fact prove a stronger comparison theorem for viscosity solutions (see \cite{BV}, Theorem 10.2):

\begin{thm}\label{cp_strong}
Let $u_1$ and $u_2$ be respectively a viscosity subsolution and a viscosity supersolution of (PME--D) with initial data $u_{0, 1}$ and $u_{0, 2}$ such that $u_{0, 1} \leq u_{0, 2}$.  Then $u_1(x, t) \leq u_2(x, t)$.
\end{thm}

We can now strength Lemma \ref{no_cross_from_above}:

\begin{lemma}\label{comp_above}
Let $u$ be a viscosity subsolution of (PME--D), and let $\varphi$ be a classical free boundary supersolution which lies above $u$ at some time $t_0$.  Then $\varphi$ cannot cross $u$ from above at a later time $t > t_0$.
\end{lemma}
\begin{proof}
Let us first replace $\varphi$ by a viscosity solution of (PME--D) with the same initial data which we denote $v$.  Since $\varphi$ is a viscosity supersolution by Lemma \ref{classical_visc}, we have that $\varphi \geq v$ by Theorem \ref{cp_strong}. Finally, $u \leq v$ by Theorem \ref{thm:cp}.
\end{proof}

%\begin{remark}
%The above result means that it is sufficient to first prove the weak comparison principle Theorem \ref{thm:cp} using \emph{only} the (classical free boundary) supersolutions from Corollary \ref{parabolic_super_barrier}, then we may make comparison arguments using \emph{any} (classical free boundary) supersolution.  \red{be careful about parabolic cylinder etc... maybe clarify?}
%\end{remark}

\section{Convergence to Equilibrium}
We begin by discussing the set of equilibrium solutions to (PME--D) and reviewing some known results.  Since by Theorem \ref{exists_unique}, the unique viscosity solution coincides with the continuous weak solution, we may invoke the results of \cite{BH} which are stated in terms of weak solutions.    

The set of equilibrium solutions and uniform convergence of solutions to the equilibrium are established in \cite{BH}. We state below a restricted version of the result for our purposes.
\begin{thm}\label{equilibrium_solution} [Theorem 5.1, \cite{BH}]
Functions of the form $(C - \Phi)_+$ with $C \in \mathbb R$ are equilibrium solutions for (PME--D), and given $u_0$, 
there exists a unique $C_\infty > 0$ such that $u(x,t)$ uniformly converges to $(C_\infty - \Phi)_+$ as $t\to\infty$.
\end{thm}

It is fairly immediate that the set $\{ (C - \Phi)_+, C \in \mathbb R\}$ is contained in the set of equilibrium solutions; the converse containment and the convergence statement are established based on a $L^1$--contraction result in \cite{BH}.  The uniqueness of the equilibrium solution comes from the fact that due to its divergence form, the density function  $\rho(\cdot,t)$ given in (\ref{density}) preserves its $L^1$--norm over time.  It follows that the equilibrium solution is determined by the condition
$$
\int [u_0]^{1/(m-1)}(x) dx = \int [u_{\infty}]^{1/(m-1)}(x) dx.
$$

We can now state our free boundary convergence theorem.

\begin{thm}\label{conv_fb}
Let $u_{\infty}(x)$ be the unique equilibrium solution given by Theorem \ref{equilibrium_solution}.  Then $\Gamma_t(u)$ uniformly converges to $\Gamma(u_\infty)$ in the Hausdorff distance, as $t\to\infty$.
\end{thm}

\begin{proof}
It will be convenient to look at the sublevel sets 
\[\Phi_{\{\leq r\}}:= \{x: \Phi(x) \leq r\}.\]
We note that since $\Phi$ is strictly convex, the sublevel sets 
 with $r \in \mathbb R$ start from a point (where the global minimum of $\Phi$ is achieved) and exhaust $\mathbb R^n$ (since $|\nabla \Phi| > 0$ except at the minimum).  Let $\e > 0$.   Invoking continuity of $\Phi$ if necessary, it is sufficient to show that there exists a $T > 0$ such that 
 \[ \Gamma_t(u) \subset \Phi_{\{\leq C_\infty + \e\}} \setminus \Phi_{\{\leq C_\infty - \e\}}\]
 for all $t \geq T$.  
 
$\Gamma_t(u) \subset [\Phi_{\{\leq C_\infty - \e\}}]^c$: This is the statement that the free boundary converges from the inside and is immediate from the uniform convergence to $u_\infty$.  Indeed, uniform convergence implies that e.g., there exists $T > 0$ such that $u(x, t) \geq \e/2$ for all $(x, t) \in \Phi_{\{\leq C_\infty - \e\}} \times [T, \infty)$; it necessarily follows then that $\Gamma_t(u)$ lies outside $\Phi_{\{ \leq C_\infty - \e\}}$ since otherwise by continuity we can find a sufficiently small ball inside $\Phi_{\{ \leq C_\infty - \e\}}$ where $u(x, t) < \e/2$, contradicting the definition of $T$.

$\Gamma_t(u) \subset \Phi_{\{\leq C_\infty + \e\}}$: In principle this should follow from a ``dual'' sort of argument, i.e., some statement to the effect that if $y \in \Gamma_t(u)$, then there exists some sufficiently small $\delta > 0$ such that $u(\cdot, t)$ is ``large'' in $B_\delta(y)$.  (E.g., an analogue of Corollary 2.2 in \cite{CF} for (PME--D) would suffice for us.)  We will instead directly make a comparison argument to show that the potential appropriately ``pulls'' the free boundary inwards, taking advantage of the explicit form of the supersolutions constructed in Corollary \ref{parabolic_super_barrier}.   

\begin{itemize}

\item The uniform convergence of $u(\cdot, t)$ to $(C_\infty - \Phi)_+$ means that there exists some continuous decreasing function $\eta(t) \rightarrow 0$ such that 
\[ u(x, t) \leq \eta(t), ~~~\mbox{for } x \notin \Phi_{\{ \leq C_\infty\}}.\] 
Let us define 
\begin{equation}\label{defn_T} T^* = \sup \left\{t: \Omega_t \subset \Phi_{ \leq C_\infty + K\tilde \eta(t)}\right\},\end{equation}
where $\tilde \eta$ is a smoothed version of $\eta$ so that it is differentiable (obtained e.g., by convolving with a small smooth bump function) and $K > 0$ is some (large) constant.
\end{itemize}

Let us first make a heuristic, non--rigorous argument to the effect that $T^* = \infty$.  Let 
\[x_* \in \partial\Phi_{\{\leq C_\infty + r(T^*)\}} \cap \Gamma_{T^*}(u); \]
such an $x_*$ exists by the minimality of $r(T^*)$, and morally it represents the worst case scenario: $x_*$ is the furthest away from where it should be among all points $x \in \Gamma_{T^*}(u)$.  We note that $\nabla \Phi(x_*)$ and $\nabla u(x_*)$ point in the outward and inward normal directions of $\Gamma_{T^*}(u)$, respectively, and hence have opposite sign.  We conclude then that the normal velocity of the boundary at $x_*$ is 
\[ V(x_*) = |\nabla u(x_*)| + \nabla \Phi \cdot \frac{\nabla u(x_*)}{|\nabla u(x_*)|} = |\nabla u(x_*)| - |\nabla \Phi(x_*)|.\]
If we assume for the moment that $|\nabla u(x_*)| \approx u(x_*) \leq \eta(T^*)$, and that $|\nabla \Phi(x_*)| \geq \Delta$, then we see that 
\[ V(x_*) \lesssim \eta(T^*) - \Delta.\]
Now if it can be arranged so that 
\[ \eta(T^*) - \Delta < \tilde \eta^\prime(t), \]
then we can conclude that the most outlying point of $\Gamma_{T^*}$ is moving inwards at a faster rate than indicated by the rate defining $T^*$, thus yielding a moral contradiction to the maximality of $T^*$.  

\begin{figure}
\center{\includegraphics[height=4in,width=4in]{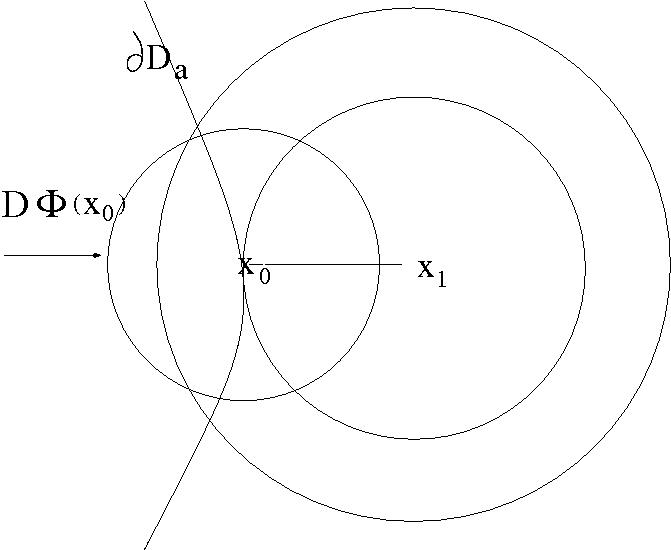}}\caption{The construction of barriers.  Here $\partial D_a = \partial \Phi_{\{ \leq C_\infty + K \tilde \eta(T^*)\}}$.}\label{fig_final_thm}
\end{figure}

A rigorous version of the argument will be established by invoking Lemma \ref{comp_above} to compare $u$ with a (classical free boundary) supersolution constructed in Corollary \ref{parabolic_super_barrier}, and thus all considerations regarding the size of $u$, $\nabla u$ and the boundary velocity will be transferred to the relevant supersolution.  These supersolutions are rescaled and translated versions of spherically symmetric supersolutions of (PME) supported in annuli.  
We remark that since we are not establishing a precise quantitative rate, the precise form given in \eqref{defn_T} is only for convenience: $\eta(t)$ sets a convenient small scale.  

\begin{itemize}
\item We will now take $x_0 \in \partial \Phi_{\{ \leq C_\infty + K\tilde \eta(T^*)\}}$ to be arbitrary.  Let us first describe the geometric setting in which we will make our comparison argument.  Since $\partial \Phi_{\leq C_\infty + K\tilde \eta(T^*)}$ is convex, there is an exterior ball at $x_0$ which is centered at some $x_1$, which we denote $B_{\alpha}(x_1)$.  The comparison argument takes place in the parabolic cylinder 
\[ \Sigma = [B_{\alpha}(x_1) \cap B_{\alpha^\prime}(x_0)] \times [T^*, T^*+\tau].\]
(See Figure \ref{fig_final_thm}.)  Here $\alpha^\prime$ is sufficiently small so that $u(x, T^*) \lesssim \eta(T^*)$ in $B_{\alpha^\prime}(x_0)$ and $\alpha$ is correspondingly small, and, finally, $\tau$ is sufficiently small in a way to be computed shortly.

\item By translation, we will assume without loss of generality that the relevant barriers are centered at $(0, 0)$ (of course, really, $t_0 = T^*$, and we shall recover $T^*$ when it is relevant) and we will denote by e.g., $\Phi = \Phi(x_0)$, with $x_0$ as before and taken to be the point at which the relevant Taylor expansion is performed.  Now returning to the proof of Corollary \ref{parabolic_super_barrier}, we see that solving the relevant minimization problem gives the explicit expression 
\[ \underline H(x, t; C\alpha) = Ae^{C\alpha t} (|x - (C\alpha(1-t)) \hat x| + \omega t - B)_+,\]
where we recall that $C$ is a constant depending on the $C^2$--norm of $\Phi$ and $B$ is a number of order unity.  Thus, 
\[ \begin{split} \tilde H(x, t) &= \alpha \underline H(\alpha^{-1}(x + \vec bt), \alpha^{-1} t; C\alpha)\\
&= A e^{C\alpha t} (|x -  C\alpha(1-t) \hat x + \vec b t| + \omega t - B\alpha)_+\\
&= A e^{C\alpha t} (~~|(|x| - C\alpha(1-t))\hat x + \vec b t| + \omega t - B\alpha~~)_+,
\end{split}\]
where $\vec b = \nabla \Phi$.  So we need to ensure that
\begin{equation}\label{fcn_mag} \begin{split} \tilde H(x, t) &\geq A e^{C\alpha t} (~~[|x| + \omega t] - [C\alpha(1 - t) + |\vec b|t + B\alpha]~~)\\
&\gtrsim \eta(T^*)
\end{split}\end{equation}
on the parabolic boundary $\partial \Sigma$.

\item
Next we compute the free boundary velocity of $\tilde H$.  We have 
\[ \begin{split}
\tilde H_t &= C\alpha \tilde H + A e^{C \alpha t} \left(\omega + \frac{d}{dt} |(|x| - C\alpha(1-t))\hat x + \vec b t| \right)\\
&= C\alpha \tilde H + Ae^{C\alpha t} \left( \omega + \left\langle \widehat{y + \vec b t}, \alpha t \hat y + \vec b \right\rangle \right),
\end{split}\]
where $y= (|x| - C\alpha(1-t)) \hat x$ is a shrunken version of $x$, whereas 
\[ \nabla \tilde H = \nabla |y + \vec b t| = \widehat{y + \vec b t} \]
has absolute value equal to unity.  Thus
\begin{equation}\label{speed_detail}\begin{split}V(\tilde H) &= \frac{\tilde H_t}{|\nabla \tilde H|}\\
&= Ae^{C\alpha t} \left( C\alpha(|y + \vec b t| + \omega t - B \alpha) + \omega + \left\langle \widehat{y + \vec b t}, \alpha t \hat y + \vec b \right\rangle\right)\\
&= Ae^{C\alpha t} \left((1+C\alpha t) \omega + C\alpha |y + \vec b t|+ \frac{|y| + \alpha t^2}{|y + \vec b t|} \langle \hat y, \vec b \rangle + t\frac{\alpha |y| + |\vec b|^2}{|y + \vec b t|} - BC\alpha^2\right).
\end{split}\end{equation}
To ensure that locally the free boundary is shrinking faster than the rate given in the definition of $T^*$ we need that 
\[ V(\tilde H) < \tilde \eta^\prime(T^*).\]

\item We will now amalgamate our conditions and observe the correct choice of parameters.  Let us first uniformly bound
\[|\nabla \Phi| \gtrsim \Delta \mbox{~~~in~~~} \left[\Phi_{\{ \leq C_\infty\}}\right]^c.\]
We will take $t \leq \tau = O(\alpha)$ and note that e.g.,  
\[|y + \vec b t| \geq | |x| - C\alpha(1-t) - \delta t| \sim | |x| - (C + \delta) \alpha | + O(\alpha^2),\]
so the expression \eqref{speed_detail} becomes (with $A = 1$)
\begin{equation*} V(\tilde H) = \omega - c \langle \hat y, \vec b \rangle  + O(\alpha^2) \geq \omega - c \Delta + O(\alpha^2), \end{equation*}
for some constant $c > 0$ of order unity.  Here the negativity of $\langle \hat y, \vec b \rangle$ comes from the fact that $\nabla \Phi$ is pointing outwards, whereas the vector from $x_1$ to $x_0$ is pointing in the opposite direction (again see Figure \ref{fig_final_thm}).  If we also take $\tilde \eta^\prime(T^*)$ to be higher order in $\alpha$, then we arrive at the condition 
\begin{equation}\label{cond_2} \omega - c \Delta < 0.\end{equation}

Similarly, taking $\eta(T^*)$ also to be higher order in $\alpha$, the condition \eqref{fcn_mag} becomes
\begin{equation}\label{cond_1}  |x| + \omega \alpha > (C + B + \delta) \alpha + O(\alpha^2),\end{equation}
where $\delta$ is some local upper bound for $|\nabla \Phi|$: $|\vec b| \leq \delta$.

Writing $|x| = \lambda \alpha$, the conditions \eqref{cond_1} and \eqref{cond_2} finally become (for $\alpha$ sufficiently small and again with $A = 1$)
\[\lambda + \omega > C + B + \delta ~~~\mbox{and~~~} c\Delta > \omega,\]
which can be satisfied by taking $\lambda$ to be a large enough factor (this corresponds to inserting a factor $\lambda$ in the relevant hyperbolic (re)scaling).  

\item Since $\partial \Phi_{\{ \leq C_\infty + K\eta(T^*)\}}$ is compact, we can take $\alpha$ to be uniform over all relevant $x_0 \in \partial \Phi_{\{ \leq C_\infty + K\eta(T^*)\}}$ (recall that $\alpha$ basically needs to be small enough to ensure that $u(x, T^*) \lesssim \eta(T^*)$ in a ball of size $\alpha^\prime \sim \alpha$) thus yielding a contradiction to the definition of $T^*$.

\item We now state the result positively.  The argument in fact shows that if $T$ satisfies the condition 
\[\mathscr G(T):~~~ \Omega_{T} \subset \Phi_{\{\leq C_\infty + K\tilde \eta(T)\}},\]
then $t$ satisfies $\mathscr G(t)$ for all $t \geq T$.  Now by comparison against an equilibrium solution of the form given in Theorem \ref{equilibrium_solution} with a sufficiently large $C$, it is the case that we can choose $K > 0$ large enough so that there is some time $T$ satisfying $\mathscr G(T)$.  We can now conclude by choosing $T$ even larger (if necessary) so that $K \tilde \eta(T) \leq \varepsilon$.

\end{itemize}

\end{proof}

\section*{\large{Acknowledgments}} 
We thank Jose Carrillo for suggesting this problem and helpful discussions and communications.  We also thank IPAM for their hospitality during the Kinetic Transport program.   
I.~C.~Kim is supported by the  NSF DMS--0700732 and a Sloan Foundation fellowship. H.~K.~Lei is supported by the UCLA Dissertation Year Fellowship, under NSF DMS--0805486 and NSF DMS--1004735.
\hspace{16 pt}

\end{document}